\documentclass[11pt]{amsart} 
\usepackage{verbatim, latexsym, amssymb, amsmath,color}
\usepackage{epsfig}

\def\N{\mathbb N}
\def\Z{\mathbb Z}

\newtheorem{thm}{Theorem}[section]
\newtheorem{lemm}[thm]{Lemma}

\newtheorem{prop}[thm]{Proposition}
\theoremstyle{remark}

\theoremstyle{definition}

\title{Density of minimal hypersurfaces for generic metrics}
\author{Kei Irie, Fernando C. Marques and Andr\'e Neves}
\address{Research Institute for Mathematical Sciences \\ Kyoto University \\ Kyoto 606-8502 \\ Japan / Simons Center for Geometry and Physics \\ State University of New York \\ Stony Brook  NY 11794-3636 \\ USA}
\email{iriek@kurims.kyoto-u.ac.jp}
\address{Princeton University \\ Fine Hall \\ Princeton NJ 08544 \\ USA}
\email{coda@math.princeton.edu}
\address{University of Chicago \\ Department of Mathematics \\ Chicago IL 60637\\ USA /Imperial College London\\ Huxley Building \\ 180 Queen's Gate \\ London SW7 2RH \\ United Kingdom}
\email{aneves@uchicago.edu, a.neves@imperial.ac.uk}
\thanks{The first author is supported by JSPS Postdoctoral
Fellowships for Research Abroad. The second author is partly supported by NSF-DMS-1509027 and NSF DMS-1311795. The third author is partly supported by NSF  DMS-1710846 and EPSRC Programme Grant EP/K00865X/1.}

\begin{document}
\maketitle

\begin{abstract}
For almost all Riemannian metrics (in the $C^\infty$ Baire sense) on a closed manifold $M^{n+1}$, $3\leq (n+1)\leq 7$, we prove that the union of all closed, smooth, embedded minimal hypersurfaces is dense.
This implies there are infinitely many minimal hypersurfaces thus proving a conjecture of Yau (1982) for generic metrics.

 \end{abstract}

\section{Introduction}

Minimal surfaces are among the most extensively studied objects in Differential Geometry. There is a wealth of examples for many particular ambient spaces, but  their general existence theory in Riemannian manifolds is still rather mysterious. A motivating conjecture has been:

\medskip

{\bf Conjecture} (Yau \cite{yau1}, 1982): \textit{Every closed Riemannian three-manifold contains infinitely many smooth, closed, immersed minimal surfaces.}

\medskip

In this paper we settle the generic case, and  in fact prove that a much stronger property holds true:  there are infinitely many closed embedded minimal hypersurfaces intersecting any given ball in $M$.

\medskip

{\bf Main Theorem}: \textit{Let $M^{n+1}$ be a closed manifold of dimension $(n+1)$, with $3\leq (n+1)\leq 7$. Then for a $C^\infty$-generic Riemannian metric $g$ on $M$, the union of all closed, smooth, embedded minimal hypersurfaces is dense.}
\medskip

Besides some specific metrics (e.g. \cite{lawson}), the existence of infinitely many closed, smooth, embedded minimal hypersurfaces  was only known   for manifolds of positive Ricci curvature  $M^{n+1}$, $3 \leq (n+1) \leq 7$, as proven by the last two authors in \cite{marques-neves-infinitely}.  Before that the best result was due to Pitts (1981, \cite{pitts}), who built on earlier work of Almgren (\cite{almgren-varifolds})
to prove there is  at least one closed embedded minimal hypersurface. In \cite{marques-neves-infinitely} it was shown the existence of at least $(n+1)$ such hypersurfaces.

The main ingredient in the proof of our Main Theorem is  the Weyl law for the volume spectrum conjectured by Gromov (\cite{gromov}) and recently proven by the last two authors jointly with Liokumovich in \cite{liokumovich-marques-neves}. We  need the Morse index estimates proven by the last two authors in \cite{marques-neves-index}, for minimal hypersurfaces constructed by min-max methods. And we use the Structure Theory of White (\cite{white2}, \cite{white}), who proved that a generic metric is ``bumpy'', meaning that every closed minimal hypersurface  is  nondegenerate. Finally, we use an idea of the first author (\cite{irie}) who proved an analogous density result for closed geodesics (not necessarily embedded) in surfaces. The argument of  \cite{irie} is based on a different kind of asymptotic law, involving spectral invariants in Embedded Contact Homology (\cite{cristofaro-hutchings-ramos}).

The volume spectrum of a compact Riemannian manifold $(M^{n+1},g)$ is a nondecreasing sequence of numbers $\{\omega_k(M,g):k\in \mathbb{N}\}$ defined variationally by performing a min-max procedure for the area (or $n$-dimensional volume) functional over  multiparameter sweepouts. The first estimates for these numbers were proven by Gromov in the late 1980s \cite{gromov0} (see also Guth \cite{guth}). 

The main result of \cite{liokumovich-marques-neves} used in this paper is:

\medskip

{\bf Weyl Law for the Volume Spectrum} (Liokumovich, Marques, Neves, 2016): {\it There exists a universal constant $a(n)>0$ such that for any compact Riemannian manifold $(M^{n+1},g)$ we have:
$$
\lim_{k\rightarrow \infty} \omega_k(M,g)k^{-\frac{1}{n+1}} = a(n) {\rm vol}(M,g)^\frac{n}{n+1}.
$$
}

In \cite{gromov0}, Gromov worked with a definition of $\omega_k(M,g)$ that was slightly different from ours (see Section 2 of this paper). He considered a parametrization of the space of hypersurfaces in $M$ by the space of real functions on $M$, or more precisely by its projectivization. Namely, to a function $f:M\rightarrow \mathbb{R}$ (or to its equivalence class $[f]$) he associated the zero set $f^{-1}(0)\subset M$. In our case, the space of hypersurfaces is the space $\mathcal{Z}_n(M;\Z_2)$ of $n$-dimensional modulo two flat boundaries endowed with the flat topology. This allows us to use the machinery of Geometric Measure Theory. The projectivization of the space of real functions can be identified immediately with $\mathbb{RP}^\infty$, while the fact that $\mathcal{Z}_n(M;\Z_2)$ is weakly homotopically equivalent to $\mathbb{RP}^\infty$ follows from work of Almgren \cite{almgren} (as explained in \cite{marques-neves-topology}). In Gromov's work (\cite{gromov}), $\omega_k$ is defined to be the smallest number such that the set of hypersurfaces with volume less than or equal to $\omega_k$ has ``essential dimension'' (Section 0.3.A, \cite{gromov}) greater than or equal to $k$. 

Flat chains modulo two of any codimension were crucially used by Guth \cite{guth} in his study of min-max volumes associated with cohomology classes. In our case we restrict to codimension one (in which case the cohomology classes are cup products) and  add the no concentration of mass condition for technical reasons related to Almgren-Pitts min-max theory.

The dimensional restriction in the Main Theorem  is due to the fact that in higher dimensions min-max (even area-minimizing) minimal hypersurfaces can have singular sets. We use Almgren-Pitts theory (\cite{almgren-varifolds}, \cite{pitts}), which together with Schoen-Simon regularity (\cite{schoen-simon}) produces smooth minimal hypersurfaces when $3\leq (n+1) \leq 7$. We expect that the methods of this paper can be generalized to handle the higher-dimensional singular case.

We finish the introduction with some idea of the proof. First we prove  that for each $k\in \mathbb{N}$, the number $\omega_k(M,g)$ is the volume of some smooth, embedded, closed minimal hypersurface, perhaps with integer multiplicities.  The possible presence of integer multiplicities is one of the reasons why constructing distinct minimal hypersurfaces is a difficult problem. 

The main observation is that the Weyl Law for the Volume Spectrum implies a mechanism to create new minimal hypersurfaces by perturbation of the metric. Suppose $g$ is a bumpy metric (a generic property by White) such that no minimal hypersurface for $g$ intersects some nonempty open set $U\subset M$. The fact that $g$ is bumpy implies there can be at most countably many minimal hypersurfaces for $g$. We consider a family of conformal deformations $g(t)=(1+th)g$ for small $t\geq 0$, where $h$ is a nonzero nonnegative function with support contained in $U$. Because the volume of $M$ goes up strictly with $t$, the Weyl Law for the Volume Spectrum tells us that for any $t>0$ some $k$-width $\omega_k$ will satisfy $\omega_k(g(t))>\omega_k(g)$, and therefore $\omega_k$ assumes uncountably many values. Because $g(t)=g$ outside $U$, for some $g(t')$, $t'>0$, there must be a minimal hypersurface that intersects $U$. Hence by perturbing $g$ to $g(t')$ we have kept all the  minimal hypersurfaces for $g$ intact but gained a new one that intersects $U$.

\section{Preliminaries}

We denote by $\mathcal{Z}_n(M;\Z_2)$ the space of modulo two $n$-dimensional flat chains $T$ in $M$  with $T=\partial U$ for some $(n+1)$-dimensional modulo two flat chain $U$ in $M$, endowed with the flat topology. This space is weakly homotopically equivalent to $\mathbb{RP}^\infty$ (see Section 4 of \cite{marques-neves-topology}). We denote by $\overline{\lambda}$ the generator of $H^1(\mathcal{Z}_n(M;\Z_2), \mathbb{Z}_2)=\mathbb{Z}_2$. The mass ($n$-dimensional volume) of $T$ is denoted by $M(T)$.

Let $X$ be a finite dimensional simplicial complex. A continuous map $\Phi:X\rightarrow \mathcal{Z}_n(M;\Z_2)$ is called a {\em  $k$-sweepout} if
$$
\Phi^*(\bar \lambda^k) \neq 0 \in H^k(X;\Z_2).
$$
We say $X$ is {\em $k$-admissible} if there exists  a $k$-sweepout $\Phi:X\rightarrow \mathcal{Z}_n(M;\Z_2)$  that has no concentration of mass, meaning 
$$\lim_{r\to 0} \sup\{M(\Phi(x) \cap B_r(p)):x\in X, p\in M\}=0.$$
The set of all $k$-sweepouts $\Phi$ that have no concentration of mass is denoted by $\mathcal P_k$. Note that two maps in  $\mathcal P_k$ can have different domains.

In \cite{marques-neves-infinitely}, the last two authors  defined
\medskip

{\bf Definition:} The {\em $k$-width of $(M,g)$}  is the number
$$\omega_k(M,g)=\inf_{\Phi \in \mathcal P_k}\sup\{M(\Phi(x)): x\in {\rm dmn}(\Phi)\},$$
where ${\rm dmn}(\Phi)$ is the domain of $\Phi$.

As remarked in the Introduction, this is a variation of a definition of Gromov (Section 4.2.B, p. 179, \cite{gromov0}).

\begin{lemm}\label{continuity}
The $k$-width $\omega_k(M,g)$ depends continuously on the metric $g$ (in the $C^0$ topology).
\end{lemm}

\begin{proof}
Suppose $g_i$ is a sequence of smooth Riemannian metrics that converges to $g$ in the $C^0$ topology. Given $\varepsilon >0$, let  $\Phi:X\rightarrow \mathcal{Z}_n(M;\Z_2)$ be a $k$-sweepout of $M$ that has no concentration of mass (this condition does not depend on the metric) and such that 
$$
\sup\{M_{g}(\Phi(x)): x \in X\} \leq \omega_k(M,g) + \varepsilon,
$$
where $M_g(T)$ is the mass of $T$ with respect to $g$.

Since 
\begin{eqnarray*}
\omega_k(M,g_i) &\leq& \sup\{M_{g_i}(\Phi(x)): x \in X\} \\
&\leq& (\sup_{v \neq 0} \frac{g_i(v,v)}{g(v,v)})^\frac{n}{2}\sup\{M_{g}(\Phi(x)): x \in X\}\\
&\leq&  (\sup_{v \neq 0} \frac{g_i(v,v)}{g(v,v)})^\frac{n}{2}(\omega_k(M,g) + \varepsilon) ,
\end{eqnarray*}
and $\varepsilon >0$ is arbitrary, we get $
\limsup_{i \rightarrow \infty} \omega_k(M,g_i) \leq \omega_k(M,g).$
Similarly, one can prove
$
\liminf_{i \rightarrow \infty} \omega_k(M,g_i) \geq \omega_k(M,g).
$

\end{proof}

The proof of the next  Proposition is essentially contained in Section 1.5 of \cite{marques-neves-index}, but we prove it here for the sake of completeness. It follows from the index estimates of the last two authors (\cite{marques-neves-index}) and a compactness theorem of Sharp (\cite{sharp}).
\begin{prop}\label{invariant.can.be.realized}
Suppose $3\leq (n+1) \leq 7$. Then for each $k\in \mathbb{N}$, there exist a finite disjoint collection $\{\Sigma_1,\dots,\Sigma_N\}$ of closed, smooth, embedded minimal hypersurfaces in $M$, and integers $\{m_1,\dots,m_N\} \subset \mathbb{N}$, such that
$$
\omega_k(M,g) = \sum_{j=1}^N m_j {\rm vol}_g(\Sigma_j),
$$
and
$$
\sum_{j=1}^N {\rm index}(\Sigma_j) \leq k.
$$
\end{prop}

\begin{proof}
Choose a sequence $\{\Phi_i\}_{i\in\N}\subset  \mathcal{P}_k$ such that
$$
\lim_{i \rightarrow \infty} \sup\{M(\Phi_i(x)): x\in X_i={\rm dmn}(\Phi_i)\} = \omega_k(M,g).
$$
Denote by $X_i^{(k)}$ the $k$-dimensional skeleton of $X_i$. Then $H^k(X_i,X_i^{(k)};\Z_2)=0$ and hence the long exact cohomology sequence gives that the natural pullback map from $H^k(X_i;\Z_2)$ into $H^k(X_i^{(k)};\Z_2)$ is injective. This implies  ${(\Phi_i)}_{|X_i^{(k)}}\in \mathcal{P}_k$. The definition of $\omega_k$ then implies 
$$
\lim_{i \rightarrow \infty} \sup\{M(\Phi_i(x)): x\in X_i^{(k)}\} = \omega_k(M,g).
$$
The interpolation machinery developed by the last two authors (\cite{marques-neves-infinitely}, item (ii) of Corollary 3.12)   implies that we can suppose $\Phi_i:X_i^{(k)} \rightarrow \mathcal{Z}_n(M,\mathbb{Z}_2)$ is continuous in the ${\bf F}$-metric (see Section 2.1 of \cite{marques-neves-infinitely}) for every $i$.

We denote by $\Pi_i$ the homotopy class of $\Phi_i$ as defined in \cite{marques-neves-index}. This is the class of all maps
$\Phi_i': X_i^{(k)} \rightarrow \mathcal{Z}_n(M,\mathbb{Z}_2)$, continuous in the ${\bf F}$-metric, that are homotopic to $\Phi_i$ in the flat topology. In particular, $(\Phi_i')^*(\overline{\lambda}^k) = \Phi_i^*(\overline{\lambda}^k)$. Continuity in the ${\bf F}$-metric implies no concentration of mass, hence every such $\Phi_i'$ is also a $k$-sweepout.

Therefore the min-max number (defined in  Section 1 of \cite{marques-neves-index})
$$
 {\bf L}(\Pi_i) = \inf_{\Phi_i' \in \Pi_i}\sup_{x\in X_i^{(k)}}\{{\bf M}(\Phi_i'(x))\}
 $$
satisfies
$$
 \omega_k(M,g) \leq {\bf L}(\Pi_i) \leq  \sup\{M(\Phi_i(x)): x\in X_i^{(k)}\}
$$
and in particular $$\lim_{i \rightarrow \infty} {\bf L}(\Pi_i) = \omega_k(M,g).$$

Theorem 1.2 of \cite{marques-neves-index} now implies the existence of a finite disjoint collection $\{\Sigma_{i,1},\dots,\Sigma_{i,N_i}\}$ of closed, smooth, embedded minimal hypersurfaces in $M$, and integers $\{m_{i,1},\dots,m_{i,N_i}\} \subset \mathbb{N}$, such that
$$
 {\bf L}(\Pi_i) = \sum_{j=1}^{N_i} m_{i,j} {\rm vol}_g(\Sigma_{i,j}),
$$
and
$$
\sum_{j=1}^{N_i} {\rm index}(\Sigma_{i,j}) \leq k.
$$

The monotonicity formula for minimal hypersurfaces in Riemannian manifolds implies that there exists  $\delta>0$, depending only on $M$, such that the volume of any closed minimal hypersurface is greater than or equal to $\delta$. Hence the number of components $N_i$ and the multiplicities $m_{i,j}$ are uniformly bounded. The Compactness Theorem of Sharp (Theorem 2.3 of \cite{sharp}) implies that there exists a finite disjoint collection $\{\Sigma_1,\dots,\Sigma_N\}$ of closed, smooth, embedded minimal hypersurfaces in $M$, satisfying
$$
\sum_{j=1}^N {\rm index}(\Sigma_j) \leq k,
$$
and integers $\{m_1,\dots,m_N\} \subset \mathbb{N}$ such that, after passing to a subsequence,
$$
\sum_{j=1}^{N_i} m_{i,j} \cdot \Sigma_{i,j} \rightarrow \sum_{j=1}^N m_j \cdot \Sigma_j
$$
as varifolds. Hence $\omega_k(M,g) = \sum_{j=1}^N m_j {\rm vol}_g(\Sigma_j)$, and the proof of the proposition is finished.

\end{proof}

\begin{prop}\label{perturbation}
Let $\Sigma$ be a closed, smooth, embedded minimal hypersurface in $(M^{n+1},g)$. Then there exists a sequence of metrics $g_i$ on $M$, $i\in \mathbb{N}$, converging to $g$ in the smooth topology such that $\Sigma$ is a nondegenerate
minimal hypersurface in $(M^{n+1},g_i)$ for every $i$.
\end{prop}

\begin{proof}
If $\tilde{g}=\exp(2\phi)g$, then the second fundamental form of $\Sigma$ with respect to $\tilde{g}$ is given by (Besse \cite{besse}, Section 1.163)
\begin{eqnarray*}\label{second.fundamental.form}
A_{\Sigma, \tilde{g}} =  A_{\Sigma,g} -  g \cdot (\nabla \phi)^\perp,
\end{eqnarray*}
where $(\nabla \phi)^\perp(x)$ is the component of $\nabla \phi$ normal to $T_x\Sigma$.  The Ricci curvatures are related by (see Besse \cite{besse}, Theorem 1.159):
\begin{eqnarray*}\label{ricci.curvature}
Ric_{\tilde{g}} = Ric_{g}-(n-1)(Hess_g\phi-d\phi \otimes d\phi) -(\Delta_g \phi +(n-1) |\nabla \phi|^2)\cdot g.
\end{eqnarray*}

Suppose both $\phi$ and $\nabla \phi$ vanish on $\Sigma$. Then $\tilde{g}_{|\Sigma} = g_{|\Sigma}$ and $A_{\Sigma, \tilde{g}}=A_{\Sigma,g}$. In particular, $\Sigma$ is also minimal with respect to $\tilde{g}$ and $|A_{\Sigma,\tilde{g}}|_{\tilde{g}}^2=|A_{\Sigma,g}|_g^2$. A unit normal $N$ to $\Sigma$ with respect to $g$ is also a unit normal to $\Sigma$ with respect to $\tilde{g}$ and
$$
Ric_{\tilde{g}}(N,N) = Ric_g(N,N)-(n-1) Hess_g\phi(N,N)-\Delta_g\phi.
$$
Since $\nabla \phi=0$ on $\Sigma$, we have $\Delta_g\phi = Hess_g\phi(N,N)$ on $\Sigma$ and therefore
$$
Ric_{\tilde{g}}(N,N) = Ric_g(N,N)-n Hess_g\phi(N,N).
$$


Let $\eta:M \rightarrow \mathbb{R}$ be a smooth function such that is equal to 1 in $V_\delta(\Sigma)$ and equal to zero in $M \setminus V_{2\delta}(\Sigma)$, where $V_r(\Sigma) = \{x\in M: d_g(x,\Sigma) \leq r\}$. We choose $\delta>0$ sufficiently small so that the function $x \mapsto d_g(x,\Sigma)^2$ is smooth in $V_{3\delta}(\Sigma)$. We define 
$h(x) = \eta(x) d_g(x,\Sigma)^2$ for $x\in V_{3\delta}(\Sigma)$ and $h(x)=0$ for $x \in M \setminus V_{3\delta}(\Sigma)$, so 
$h:M \rightarrow \mathbb{R}$ is a smooth function  that coincides with $x \mapsto d_g(x,\Sigma)^2$ in some small neighborhood of $\Sigma$.

Let $g_i = \exp(2\phi_i)g$, where $\phi_i=\frac{1}{i}h$. Since $h(x)=d_g(x,\Sigma)^2$ in a neighborhood of $\Sigma$, we have that, on $\Sigma$, $\phi_i=0$, $\nabla \phi_i=0$ and $Hess_g \phi_i(N,N) = \frac{2}{i}$, and $\Sigma$ is minimal with respect to $g_i$.

Therefore
$$
Ric_{g_i}(N,N) + |A_{\Sigma,g_i}|_{g_i}^2 = Ric_g(N,N) + |A_{\Sigma,g}|_g^2  -\frac{2n}{i}.
$$

The Jacobi operator acting on normal vector fields is given by the expression
$$
L_{\Sigma,g}(X) = \Delta_{\Sigma,g}^\perp X + (Ric_g(N,N)+ |A_{\Sigma,g}|_g^2) X.
$$
Since  ${g_i}_{|\Sigma} = g_{|\Sigma}$, we have $\Delta_{\Sigma,g_i}^\perp X = \Delta_{\Sigma,g}^\perp X$ and hence
$$
L_{\Sigma,g_i}(X) =L_{\Sigma,g}(X)-\frac{2n}{i}X.
$$

We conclude that
$$
{\rm spec\,}(L_{\Sigma,g_i}) = {\rm spec\,}(L_{\Sigma,g}) +\frac{2n}{i}.
$$
Hence $\Sigma$ is nondegenerate with respect to $g_i$ for every sufficiently large $i$.
\end{proof}

\section{Proof of the Main Theorem}

We denote by $\mathcal{M}$ the space of all smooth Riemannian metrics on $M$, endowed with the $C^\infty$ topology.

\begin{prop}\label{open.dense}
Suppose $3\leq (n+1) \leq 7$, and let $U \subset M$ be a nonempty open set. Then the set $\mathcal{M}_U$ of all smooth Riemannian metrics on $M$ such that there exists a nondegenerate, closed, smooth, embedded, minimal hypersurface $\Sigma$ that intersects $U$ is open and dense in the $C^\infty$ topology.
\end{prop}

\begin{proof}
Let  $g\in \mathcal{M}_U$ and $\Sigma$ be like in the statement of the proposition. Because $\Sigma$ is nondegenerate, an application of the Inverse Function Theorem implies that for every Riemannian metric $g'$ sufficiently close to $g$, there exists a unique nondegenerate closed, smooth, embedded minimal hypersurface $\Sigma'$ close to $\Sigma$. This follows, for instance, from the Structure Theorem of White (Theorem 2.1 in \cite{white2}) since the nondegeneracy of $\Sigma$ is equivalent to the invertibility of  $D\Pi(g,\Sigma)$ (here $\Pi$ is as  in \cite{white2}). In particular,   $\Sigma'\cap U \neq \emptyset$ if $g'$ is sufficiently close to $g$.  This implies $\mathcal{M}_U$ is open.

It remains to show the set $\mathcal{M}_U$ is dense. Let $g$ be an arbitrary smooth Riemannian metric on $M$ and $\mathcal{V}$ be  an arbitrary neighborhood of $g$ in the $C^\infty$ topology.  By the Bumpy Metrics Theorem of White (Theorem 2.1, \cite{white}), there exists $g'\in \mathcal{V}$ such that every  closed, smooth immersed minimal hypersurface with respect to $g'$ is nondegenerate. If one of these minimal hypersurfaces is embedded and intersects $U$ then $g'\in \mathcal{M}_U$, and we are done.

Hence we can suppose that every closed, smooth, embedded minimal hypersurface with respect to $g'$ is contained in the complement of $U$. Since $g'$ is bumpy, it follows from Sharp (Theorem 2.3 and Remark 2.4, \cite{sharp}) that the set of connected, closed, smooth, embedded minimal hypersurfaces in $(M,g')$ with both area and index bounded from above by $q$ is finite for every $q>0$. Therefore the set
\begin{eqnarray*}
&&\mathcal{C} = \{\sum_{j=1}^N m_j {\rm vol}_{g'}(\Sigma_j): N \in \mathbb{N}, \{m_j\}_{j=1}^N \subset \mathbb{N}, \{\Sigma_j\}_{j=1}^N {\rm  \, disjoint\, collection\, }\\
&& \hspace{1.5cm} {\rm of\, closed,\, smooth,\, embedded\, minimal\, hypersurfaces\, in\,} (M,g')\}
\end{eqnarray*}
is countable.

Choose $h:M \rightarrow \mathbb{R}$ a smooth nonnegative function such that ${\rm supp\,}(h) \subset U$ and $h(x)>0$ for some $x\in U$. Define $g'(t) = (1+th)g'$ for $t\geq 0$, and let $t_0>0$ be sufficiently small so that $g'(t)\in \mathcal{V}$ for every $t\in [0,t_0]$. Notice that $g'(t)=g'$ outside some compact set $K\subset U$ for every $t>0$.

We have ${\rm vol}(M,g'(t_0))> {\rm vol}(M,g')$. It follows from the Weyl Law for the Volume Spectrum (see Introduction) that there exists $k\in \mathbb{N}$ such that $\omega_k(M,g'(t_0))>\omega_k(M,g')$. Assume by contradiction that  for every $t\in [0,t_0]$, every closed, smooth, embedded minimal hypersurface in
$(M,g'(t))$ is contained in $M \setminus U$.  Since $g'(t)=g'$ outside $K \subset U$ we conclude from Proposition \ref{invariant.can.be.realized} that $\omega_k(M,g'(t)) \in \mathcal{C}$ for all $t\in [0,t_0]$. But $\mathcal{C}$ is countable and we know from Proposition \ref{continuity} that the function $t \mapsto \omega_k(M,g'(t))$ is continuous. Hence $t \mapsto \omega_k(M,g'(t))$ is constant in the interval $[0,t_0]$. This contradicts the fact that $\omega_k(M,g'(t_0))>\omega_k(M,g')$.

Therefore we can find $t \in [0,t_0]$ such that there exists a closed, smooth, embedded minimal hypersurface $\Sigma$ with respect to $g'(t)$ that intersects $U$. Since $g'(t) \in \mathcal{V}$, Proposition \ref{perturbation} implies there exists a Riemannian metric $g''\in \mathcal{V}$ such that $\Sigma$ is minimal and nondegenerate with respect to $g''$. Therefore $g''\in \mathcal{V} \cap \mathcal{M}_U$ and we have finished the proof of the Proposition.

\end{proof} 

\begin{proof}[Proof of the Main Theorem]
Let $\{U_i\}$ be a countable basis of $M$. Since,  by Proposition \ref{open.dense}, each $\mathcal{M}_{U_i}$ is open and dense in $\mathcal{M}$ the set
$\cap_i \mathcal{M}_{U_i}$ is $C^\infty$ Baire-generic in $\mathcal{M}$. This  finishes the proof.
\end{proof}

\bibliographystyle{amsbook}

\begin{thebibliography}{99}

\bibitem{almgren} 
F. Almgren, \textit{The homotopy groups of the integral cycle groups,} Topology  (1962), 257--299. 

\bibitem{almgren-varifolds}
F. Almgren, \textit{The theory of varifolds,} Mimeographed notes, Princeton (1965).

\bibitem{besse}
Besse, A. L.,
\textit{Einstein manifolds,}
Classics in Mathematics, Springer-Verlag (1987).

\bibitem{cristofaro-hutchings-ramos}
Cristofaro-Gardiner, D., Hutchings, M., Ramos, V. Gripp Barros,
\textit{The asymptotics of ECH capacities,} 
Invent. Math. 199 (2015), no. 1, 187--214.

\bibitem{gromov0} M. Gromov, \textit{Dimension, nonlinear spectra and width,} Geometric aspects of functional analysis,(1986/87), 132--184, 
Lecture Notes in Math., 1317, Springer, Berlin, 1988. 

\bibitem{gromov} M. Gromov, \textit{Isoperimetry of waists and concentration of maps,} Geom. Funct. Anal. 13 (2003), 178--215.

\bibitem{guth} L. Guth, \textit{Minimax problems related to cup powers and Steenrod squares,}  
Geom. Funct. Anal. 18 (2009), 1917--1987. 

\bibitem{irie}
Irie, K.,
\textit{Dense existence of periodic Reeb orbits and ECH spectral invariants,}
J. Mod. Dyn. 9 (2015), 357--363.

\bibitem{lawson}
Lawson, H. B., Jr.
\textit{Complete minimal surfaces in $S^3$},
Ann. of Math. (2) 92 1970 335--374. 

\bibitem{liokumovich-marques-neves}
Liokumovich, Y., Marques, F. C., Neves, A.,
\textit{Weyl law for the volume spectrum,}
arXiv:1607.08721 [math.DG] (2016), to appear in Annals of Mathematics.

\bibitem{marques-neves-index}
Marques, F. C., Neves, A.,
\textit{Morse index and multiplicity of min-max minimal hypersurfaces,}
Camb. J. Math. \textbf{4} (2016), no. 4, 463--511. 

\bibitem{marques-neves-topology}
Marques, F. C., Neves, A.,
\textit{Topology of the space of cycles and existence of minimal varieties,} 
 Surv. Differ. Geom., 21, Int. Press, Somerville, MA, (2016),   165--177,

\bibitem{marques-neves-infinitely}
Marques, F. C., Neves, A.,
\textit{Existence of infinitely many minimal hypersurfaces in positive Ricci curvature,}
Invent. Math. \textbf{209} (2017), no.2, 577--616.

\bibitem{pitts} 
J. Pitts, \textit{Existence and regularity of minimal surfaces on Riemannian manifolds,} Mathematical Notes 27, Princeton University Press, Princeton, (1981).

\bibitem{schoen-simon} 
R. Schoen and L. Simon, \textit{Regularity of stable minimal hypersurfaces.} 
Comm. Pure Appl. Math. 34 (1981), 741--797. 



\bibitem{sharp}
Sharp, B.,
\textit{Compactness of minimal hypersurfaces with bounded index,} 
J. Differential Geom. 106 (2017), no. 2, 317--339. 

\bibitem{white2}
B. White, \textit{The space of minimal submanifolds for varying Riemannian metrics,} Indiana Univ. Math. J. \textbf{40} (1991), 161--200.

\bibitem{white}
White, B.,
\textit{On the bumpy metrics theorem for minimal submanifolds,}
Amer. J. Math. 139 (2017), no. 4, 1149--1155.

\bibitem{yau1} Yau, S.-T., \textit{
Problem section.} Seminar on Differential Geometry, pp. 669Ð706, 
Ann. of Math. Stud., 102, Princeton Univ. Press, Princeton, N.J., 1982.






\end{thebibliography}

\end{document}